\documentclass[sn-mathphys]{sn-jnl}

\usepackage{amsmath}
\usepackage{amsfonts}
\usepackage{amssymb}
\usepackage[capitalize]{cleveref}
\crefformat{equation}{(#2#1#3)}
\crefmultiformat{equation}{(#2#1#3)}%
{ and~(#2#1#3)}{, (#2#1#3)}{ and~(#2#1#3)}

\newcommand\xC{{\mathbb{C}}}
\newcommand\abs[1]{\left\lvert#1\right\rvert}
\newcommand\dabs[1]{\lvert#1\rvert}
\newcommand\norm[1]{\left\|#1\right\|}
\newcommand\dnorm[1]{\|#1\|}
\DeclareMathOperator*{\argmin}{arg\,min}
\DeclareMathOperator*{\argmax}{arg\,max}

\jyear{2023}%

\theoremstyle{thmstyleone}%
\newtheorem{theorem}{Theorem}
\newtheorem{proposition}[theorem]{Proposition}%

\theoremstyle{thmstyletwo}%

\theoremstyle{thmstylethree}%

\begin{document}

\shortauthor{D. Pradovera}
\title{Toward a certified greedy Loewner framework with minimal sampling}


\author*{\fnm{Davide} \sur{Pradovera}}\email{davide.pradovera@univie.ac.at}

\affil{\orgdiv{Department of Mathematics}, \orgname{University of Vienna}, \orgaddress{\street{Oskar-Morgenstern-Platz 1}, \city{1090 Vienna}, \country{Austria}}}

\abstract{We propose a strategy for greedy sampling in the context of non-intrusive interpolation-based surrogate modeling for frequency-domain problems. We rely on a non-intrusive and cheap error indicator to drive the adaptive selection of the high-fidelity samples on which the surrogate is based. We develop a theoretical framework to support our proposed indicator. We also present several practical approaches for the termination criterion that is used to end the greedy sampling iterations. To showcase our greedy strategy, we numerically test it in combination with the well-known Loewner framework. To this effect, we consider several benchmarks, highlighting the effectiveness of our adaptive approach in approximating the transfer function of complex systems from few samples.}



\maketitle
\thispagestyle{empty}

\section{Introduction}\label{sec:intro}

Rational approximation techniques have been applied with great success in the field of surrogate modeling for frequency responses of dynamical systems, with notable applications in the synthesis of electric circuits, and in the structural response analysis for mechanical systems. In the field of \emph{model order reduction} (MOR), a plethora of methods and algorithms have been developed with this aim in the last decades, including vector fitting (VF) \cite{VF1,VF2}, the Loewner framework (LF) \cite{LF1,LF2}, RKFIT \cite{RK1}, AAA \cite{AAA}, and minimal rational interpolation (MRI) \cite{MRI}. In such methods, one leverages \emph{a priori} knowledge on the structure of the problem, namely, the rational dependence of the transfer function $H(z)$ on the frequency $z\in\xC$. For this reason, a rational function is introduced as \emph{ansatz} for the surrogate model. Its coefficients (the ``degrees of freedom'' of the surrogate) are then optimized by fitting the surrogate to the available data, usually in the form of transfer function samples: $\{(z_j,H(z_j))\}_{j=1}^S$.

Alternative approaches like reduced basis and balanced truncation have also been shown to be very effective \cite{MOR1,MOR2,MOR3,MOR4}. However, they have the disadvantage of being \emph{intrusive}, i.e., they require access to the matrices defining the dynamical system. This drastically reduces their range of application. For instance, the system matrices are unavailable when closed-source software is used to model the target transfer function. Similarly, intrusive methods cannot be applied when the transfer function is obtained via lab measurements, since, in this case, the system matrices are not even defined. In these settings, non-intrusive rational-approximation-based methods are paramount for obtaining surrogate models for the frequency response.

In order to guarantee an effective and accurate approximation, it is crucial to judiciously choose the number and locations of the sample points $\{z_j\}_{j=1}^S$ used to build the surrogate model. To this aim, it is useful to consider the questions: how many sample points are necessary to attain a desired approximation accuracy? And where should they be placed?

In this work, we look at these questions from a practical ``algorithmic'' viewpoint (thus ignoring ``optimality'' concerns that are more pertinent to information theory), following the well-explored idea of \emph{greedy} MOR \cite{MOR4}. Specifically, assume that a starting set of $S$ sampled frequencies $\{z_j\}_{j=1}^S$ is given. We want to devise a \emph{constructive} strategy for identifying the location of \emph{one} new point $z_{S+1}$, in such a way that the surrogate built from samples at $\{z_j\}_{j=1}^{S+1}$ is (\emph{a priori}) the ``best'' among all possible surrogates built in this way. Otherwise said, given only the available samples, we wish to identify the unexplored frequency that will provide the most information on the not-yet-understood features of the target $H$. We summarize this idea in \cref{algo}, where, for simplicity, we initialize the sample set as a single point.

\begin{algorithm}[htb]
\caption{Generic surrogate modeling with greedy sampling}\label{algo}
\begin{algorithmic}[1]
\State $\textbf{choose}\text{ the first sampled frequency }z_1\text{ and }\textbf{compute }H(z_1)$
\For{$S=1,2,\ldots$}
\State $\textbf{build}\text{ the surrogate }\widetilde{H}\text{ from the available }S\text{ samples}$
\If{$\widetilde{H}\text{ is accurate enough}$}
	\State $\textbf{return }\widetilde{H}$
\EndIf
\State $\textbf{find}\text{ the next point }z_{S+1}\text{ using a greedy criterion}$
\State $\textbf{compute }H(z_{S+1})\text{ and }\textbf{add }z_{S+1}\text{ to the sampled frequencies}$
\EndFor
\end{algorithmic}
\end{algorithm}

In intrusive MOR for frequency-domain problems, the greedy sampling is commonly carried out by placing the next sample point at the location where the current approximation \emph{residual} is largest. However, evaluating residuals is generally impossible in non-intrusive settings, and an alternative criterion for the greedy selection of sample points must be found. In this work, we derive one such criterion in \cref{sec:greedy}. Specifically, with our proposed indicator, we aim at approximating the \emph{relative approximation error}, a quantity that is notoriously difficult to pinpoint. This is especially the case in frequency-domain applications, since, in general, the \emph{inf-sup constant} of the problem is not uniformly lower-bounded due to resonant behavior.

We note that, throughout this paper, we use the term ``adaptivity'' in a way that differs slightly from the customary one in rational approximation. Indeed, methods like RKFIT and AAA (whose first ``A'' actually stands for ``adaptive'') are able to adapt the so-called rational \emph{type} of the surrogate, i.e., essentially, the degree of the approximation. VF has also often been endowed with strategies for adaptive selection of the rational type, which can be summarized as finding a balance between ``interpolation'' and ``generalization'' errors \cite{gVF}. However, all the above-mentioned methods for non-intrusive surrogate modeling assume that a sample set is provided \emph{in advance}, over a grid of sampled frequencies $\{z_j\}_{j=1}^S$ that is \emph{fine enough} for the identification of the relevant features of the transfer function $H$. Building this initial database of $S$ samples can be quite expensive, since it requires many queries of the high-fidelity transfer function $H$. Also, practitioners usually have to \emph{guess} a large \emph{enough} value $S$. In this work, we endeavor to remove this initial ``fine sampling'' phase. Instead, as exemplified by \cref{algo} (and more in line with intrusive greedy MOR), we seek a strategy where the sample set is parsimoniously enriched as $\widetilde{H}$ is built.

The only prior works (that we could find) on comparable approaches for non-intrusive frequency-domain MOR are \cite{gLF1,gLF2}. Also, \cite{gLF3,gLF4} explore similar ideas in intrusive settings. In the above references, the sampling is incremental and constructive. However, the greedy criteria driving the respective algorithms are only heuristic. In contrast, here we endeavor to provide more rigorous strategies, generalizing those developed for MRI in \cite{gMRI,geMRI}.

Before proceeding, we believe it important to mention the issue of noise: the samples $H(z_j)$ might be inexact. This happens whenever the samples come from lab experiments, since they might be affected by measurement error. Also, samples are noisy when approximations (e.g., finite-element discretizations) are introduced in describing the system of interest. This step is necessary when modeling, for instance, distributed electrical devices. In such cases, the noise arises as a ``discretization'' error, which cannot always be neglected \cite{adMRI}. In all these settings, interpolatory approaches might incur \emph{overfitting}, resulting in wildly inaccurate surrogate models, e.g., due to Runge oscillations. More robust least-squares approaches are usually preferred in such cases. However, in any such framework (which includes oversampling as a measure against noise), a ``truly parsimonious'' adaptivity is usually impossible. We refer to \cite{noise} for a discussion on this issue.

\section{Basics of rational MOR}\label{sec:rational}
We target a frequency-domain linear time-invariant descriptor system
\begin{equation}\label{eq:lti}
\begin{cases}
zEX(z)=AX(z)+BU(z)\\
Y(z)=CX(z)
\end{cases}
\end{equation} 
where $E,A\in\xC^{n\times n}$, $B\in\xC^{n\times m}$, and $C\in\xC^{p\times n}$. The frequency $z$ is a complex scalar number, upon which the input $U(z)\in\xC^m$, the state $X(z)\in\xC^n$, and the output $Y(z)\in\xC^p$ depend. For all $z\in\xC$, we define the associated \emph{transfer functions} for state and output, respectively, as
\begin{equation}\label{eq:ltit}
\begin{cases}
G(z)=\Phi(z)B\in\xC^{n\times m}\\
H(z)=C\Phi(z)B\in\xC^{p\times m}
\end{cases} \quad\text{with }\Phi(z)=(zE-A)^{-1}\in\xC^{n\times n}.
\end{equation}
(For conciseness, we refer to $H$ as just ``transfer function''.) The frequency values for which $H$ is unbounded (due to $zE-A$ being singular) are the \emph{poles} or \emph{resonating frequencies} of the system. In this paper, we make use of the symbol $\dnorm{\cdot}$ to denote the Frobenius norm for matrices.

As outlined above, we are interested in building an approximation of $H$ using samples of it at some frequency values: $\{(z_j,H(z_j))\}_{j=1}^S$. Throughout the paper, $S$ denotes the total number of available samples. In applications, the system dimension $n$ is often large, and computing samples of the transfer function $H$ is expensive due to the matrix inversion in $\Phi$. For this reason, we try to keep $S$ as low as possible.

We note that approximations of the state transfer function $G$ are also possible. This case can be simply recovered by setting $C=I_n$, the identity matrix of size $n$. Such setting is sometimes of interest when \cref{eq:lti} is obtained by discretizing a PDE, e.g., the Helmholtz equation. In this case, the state $X$ corresponds to the (discretized) ``solution field'' of the PDE.

In non-intrusive data-driven MOR methods, a rational approximation of the frequency response is built by fitting the sampled values of $H$. Both interpolation-based (e.g., LF \cite{LF1} and MRI \cite{MRI}) and least-squares-based (e.g., VF \cite{VF1}) approaches are available in the literature.

All these approaches start by choosing a suitable format for rational functions. The classical choice involves expanding numerator and denominator in terms of a polynomial basis, e.g., monomials or Legendre polynomials:
\begin{equation}\label{eq:rational}
\widetilde{H}(z)=\frac{\sum_{i=0}^M\widetilde{P}_i\phi_i(z)}{\sum_{i=0}^N\widetilde{q}_i\phi_i(z)},
\end{equation}
for some yet-to-be-determined sets of coefficients $\{\widetilde{P}_i\}_{i=0}^M\subset\xC^{p\times m}$ and $\{\widetilde{q}_i\}_{i=0}^N\subset\xC$. More recently, this format has been (mostly) set aside due to its limited stability properties, in favor of the \emph{barycentric format} \cite{bary,AAA}:
\begin{equation*}
\widetilde{H}(z)=\sum_{j=1}^{\textup{max}(M,N)+1}\frac{P_j}{z-\zeta_j}\Bigg/\sum_{j=1}^{\textup{max}(M,N)+1}\frac{q_j}{z-\zeta_j}.
\end{equation*}
The (distinct) \emph{support} points $\{\zeta_j\}_j\subset\xC$ can be chosen arbitrarily. However, one should be careful about defining the value of the surrogate at such points, since both numerator and denominator are unbounded there. To solve this issue, one needs to eliminate a ``removable discontinuity'':
\begin{equation}\label{eq:baryval}
\widetilde{H}(\zeta_i):=\lim_{z\to\zeta_i}\widetilde{H}(z)=\lim_{z\to\zeta_i}\frac{P_i/(z-\zeta_i)+\cdots}{q_i/(z-\zeta_i)+\cdots}=\frac{P_i}{q_i}.
\end{equation}

This leads to a rational format that is equivalent to the usual one in \cref{eq:rational}. However, the barycentric format is (usually) much more reliable and convenient, since many operations that are numerically unstable can be carried out more safely. For instance, given $S$ samples $\{(z_j,H(z_j))\}_{j=1}^S$, the \emph{interpolatory} barycentric expansion
\begin{equation}\label{eq:baryint}
\widetilde{H}(z)=\sum_{j=1}^S\frac{q_jH(z_j)}{z-z_j}\ \Bigg/\sum_{j=1}^S\frac{q_j}{z-z_j},\quad\text{with }q_j\neq 0\ \forall j\geq 1,
\end{equation}
characterizes \emph{all} rational functions of type\footnote{The type of a rational function is a pair of numbers denoting the degrees of numerator and denominator, respectively. In our presentation, we do not worry about whether the type is ``exact'' or not: it suffices for numerator and denominator to have degrees \emph{smaller than or equal to} those prescribed by the type.} $[S-1/S-1]$ that interpolate the data exactly, cf.~\cref{eq:baryval}. Notably, interpolation is achieved without involving any ill-conditioned Vandermonde matrix. Instead, to achieve interpolation in the barycentric form, one just needs to choose the sample points as support points! We refer to \cite{bary} for more details on the barycentric format.

We restrict our focus to interpolation-based approaches, where all samples are reconstructed exactly: $\widetilde{H}(z_j)=H(z_j)$ for all $j=1,\ldots,S$. Without loss of generality, we can assume that the surrogate $\widetilde{H}$ is represented in interpolatory barycentric form \cref{eq:baryint}. The coefficients $q_1,\ldots,q_S$ are found by imposing additional conditions, which depend on the specific MOR method. We mention here some options:
\begin{itemize}
\item \emph{LF for SISO systems ($p=m=1$).} The barycentric coefficients are found by imposing interpolation conditions at additional frequencies $z'_1,\ldots,z'_{S-1}$. Note that only $S-1$ additional points are necessary, since one of the degrees of freedom of the barycentric coefficients is fixed by prescribing a normalization constraint, usually $\sum_{j=1}^S\dabs{q_j}^2=1$.
\item \emph{LF for MIMO systems ($pm>1$), least-squares LF, and AAA.} The barycentric coefficients are found by solving a (linearized) least-squares fitting problem on a set of test frequencies:
\begin{equation}\label{eq:mimolf}
\textup{minimize }\sum_{l=1}^{S_{\textup{test}}}\norm{\sum_{j=1}^Sq_j\frac{H(z_l')-H(z_j)}{z_l'-z_j}}^2\quad\textup{subject to }\sum_{j=1}^S\abs{q_j}^2=1.
\end{equation}
(Note that the above-mentioned SISO LF is just a special case of this more general setting.) For MIMO LF, alternative strategies based on \emph{tangential interpolation} of $H$ are also available \cite{LF1}.
\item \emph{MRI.} The barycentric coefficients are found by solving a minimization problem involving only the original $S$ sampled frequencies. Specifically,
\begin{equation*}
\textup{minimize }\norm{\sum_{j=1}^Sq_jH(z_j)}^2\quad\textup{subject to }\sum_{j=1}^S\abs{q_j}^2=1.
\end{equation*}
MRI does not require any extra samples in addition to the initial $S$ ones. However, spurious solutions might arise if the size of $H$, namely, $pm$, is too small, since $\sum_{j=1}^Sq_jH(z_j)\equiv 0$ can be trivially satisfied. For this reason, MRI is mostly recommended only for state approximation \cite{MRI}.
\end{itemize}

\subsection{An ideal greedy strategy}\label{sec:ideal}
When approximating the transfer function $H$, an ``ideal'' greedy algorithm is driven by the $\delta$-adjusted relative error
\begin{equation*}
\varepsilon(\widetilde{H},z):=\frac{\dnorm{\widetilde{H}(z)-H(z)}}{\norm{H(z)}+\delta}.
\end{equation*}
Above, $\delta$ is a small number (e.g., $\delta=10^{-8}$) that prevents division by 0 whenever the relative error is computed at a zero of the transfer function. Our discussion mostly generalizes to $\delta=0$ (relative error) and $\delta\to\infty$ (absolute error). See \cref{sec:conclusion}.

In this setting, the next sample point is $z_{S+1}:=\argmax_z\varepsilon(\widetilde{H},z)$, and the termination criterion is
\begin{equation}\label{eq:termination}
\varepsilon_\infty:=\max_z\varepsilon(\widetilde{H},z)<\textup{tol},
\end{equation}
with $\textup{tol}$ being a user-specified tolerance. Note that this is a very idealized strategy: in general, the maximum relative error $\varepsilon_\infty$ is infinite, since the surrogate $\widetilde{H}$ is unbounded near the surrogate poles.

In practice, the continuous maximization over $z$ is replaced by its discrete version over a fine grid of $S'$ \emph{test frequencies}. In this context, the next sample point is chosen among the $S'$ test frequencies, and the termination condition reads: all the $S'$ test values of the error must be below the tolerance.

This approach is extremely expensive, since evaluating the error at all test frequencies requires $S'$ high-fidelity samples of the transfer function. For this reason, in projection-based approaches like reduced basis, it is common to employ residual-based error estimators instead of the actual error.

Such estimators cannot be developed for the transfer function, since no ``residual'' can be defined for $H$. Instead, one commonly considers the state $X$ or its transfer function $G$. Specifically, in this latter case, we define the \emph{relative residual norm} as
\begin{equation}\label{eq:residual}
\rho(\widetilde{G},z):=\frac{\dnorm{(zE-A)\widetilde{G}(z)-B}}{\norm{B}},
\end{equation}
with $\widetilde{G}$ being the current surrogate model for $G$. Note that $\rho(G,\cdot)\equiv 0$. By relying on the so-called (in MOR) \emph{affine} form of the original system \cref{eq:lti}, such residual can be efficiently computed at any frequency $z$, with a cost that is independent of the original system size $n$. See \citep{MOR4} for more details on this.

In a non-intrusive framework, residual-based estimators like \cref{eq:residual} are unavailable: since it might be impossible to gather samples of the state $X$ or of the state transfer function $G$, which are necessary to build $\widetilde{G}$ and to evaluate $\rho$. However, we can still use \cref{eq:residual} to our advantage, as we do in the next section.

\section{Non-intrusive error maximization}\label{sec:greedy}
As showcased in \cref{algo}, the greedy algorithm requires two main ingredients: the identification of the next sample point and the termination criterion check. In this section, we focus on the former. In this context, we have described above how the next sample point should be selected as the location where the approximation error or residual norm is maximal. To make this process non-intrusive, we need to express such error measures only in terms of available quantities. To this effect, we show the following property, which generalizes a similar result presented in \cite{MRI} for state approximation.

\begin{proposition}\label{prop:resequiv}
Consider an interpolatory rational approximation of $G$, namely,
\begin{equation}\label{eq:statesurr}
\widetilde{G}(z)=\sum_{j=1}^S\frac{q_jG(z_j)}{z-z_j}\Bigg/\sum_{j=1}^S\frac{q_j}{z-z_j}.
\end{equation}
There exists a $z$-independent constant $\gamma$ such that the relative residual norm satisfies
\begin{equation}\label{eq:resequiv}
\rho(\widetilde{G},z)=\gamma\abs{\sum_{j=1}^S\frac{q_j}{z-z_j}}^{-1}.
\end{equation}
\end{proposition}
\begin{proof}
Let $Q(z)=\sum_{j=1}^Sq_j/(z-z_j)$ be the surrogate barycentric denominator, and $r(\widetilde{G},z)=(zE-A)\widetilde{G}(z)-B$ the residual. By multiplying $r$ by $Q(z)$, we obtain
\begin{align*}
Q(z)r(\widetilde{G},z)=&Q(z)\left((zE-A)\sum_{j=1}^S\frac{q_jG(z_j)}{z-z_j}\bigg/Q(z)-B\right)\\
=&\sum_{j=1}^S\frac{q_j(zE-A)G(z_j)}{z-z_j}-BQ(z)=\sum_{j=1}^Sq_j\frac{(zE-A)G(z_j)-B}{z-z_j}.
\end{align*}
Now we add and subtract $z_jEG(z_j)$ from the numerator above and exploit the definition of $G(z_j)$, cf.~\cref{eq:ltit}:
\begin{align}
Q(z)r(\widetilde{G},z)=&\sum_{j=1}^Sq_j\frac{(zE+(z_j-z_j)E-A)G(z_j)-B}{z-z_j}\nonumber\\
=&\sum_{j=1}^Sq_j\frac{(z-z_j)EG(z_j)}{z-z_j}+\sum_{j=1}^Sq_j\frac{(z_jE-A)G(z_j)-B}{z-z_j}\nonumber\\
=&\sum_{j=1}^Sq_jEG(z_j).\label{eq:resvect}
\end{align}
The claim follows by taking the Frobenius norm and dividing by $\dabs{Q(z)}\dnorm{B}$. Specifically, $\gamma$ is the norm of the right-hand-side of \cref{eq:resvect}, divided by $\dnorm{B}$.
\end{proof}

The above property states that the approximation residual for the state transfer function $G$ \emph{equals}, up to a $z$-independent proportionality constant, the inverse of the magnitude of the surrogate barycentric denominator $Q(z)=\sum_{j=1}^Sq_j/(z-z_j)$. Specifically, the residual norm is large if and only if the surrogate denominator $Q$ is small.

Note that, for \cref{prop:resequiv} to hold, we just require the function $G$ to be a rational function of type $[S-1/S-1]$, interpolating the exact $G$ at all the $S$ support points. We can leverage this to obtain some useful insights on the approximation quality of non-intrusive methods.

Indeed, consider the surrogate $\widetilde{H}$ obtained by a generic interpolatory approximation-based method like LF or AAA, and assume that $q_1,\ldots,q_S$ are its barycentric coefficients, cf.~\cref{eq:baryint}. In the barycentric expansion of $\widetilde{H}$, namely, \cref{eq:baryint}, we can formally replace all occurrences of the transfer function $H$ with the state transfer function $G$, leading to a new surrogate $\widetilde{G}$ as in \cref{eq:statesurr}, which is an approximation of $G$. (Note that such operations might be impossible in practice, since samples of $G$ are generally unavailable! However, we are only proceeding formally here.) \cref{prop:resequiv} reveals the exact form of the residual norm \cref{eq:residual} of such new surrogate.

Specifically, assume that the residual norm \cref{eq:residual} is used to drive \cref{algo}. Then \cref{eq:resequiv} suggests that the approximation quality is best improved by adding new sample points at frequencies that \emph{minimize} the magnitude of the denominator $Q(z)$ appearing in the surrogate $\widetilde{G}$. Such denominator \emph{is} computable even in a non-intrusive setting. In fact, it is precisely the same denominator appearing in the surrogate $\widetilde{H}$!

The resulting strategy is summarized by the rule:
\begin{equation}\label{eq:nextptQ}
z_{S+1}:=\argmax_z\rho(\widetilde{G},z)=\argmin_z\abs{\sum_{j=1}^S\frac{q_j}{z-z_j}}.
\end{equation}
We note that, in order to apply \cref{eq:nextptQ}, it is not necessary to compute the scaling constant $\gamma$. This is crucial for non-intrusive methods, since $\gamma$ cannot be estimated non-intrusively.

Before proceeding, we mention that we can employ the standard link between error and residual to obtain a relative \emph{error} bound for $\widetilde{G}$. Specifically, let $\mathcal{K}(\cdot)$ denote the (Euclidean) condition number of a matrix, i.e., the ratio of its largest and smallest singular values. Then,
\begin{equation*}
\frac{\dnorm{\widetilde{G}(z)-G(z)}}{\norm{G(z)}}\leq \gamma\mathcal{K}(zE-A)\abs{\sum_{j=1}^S\frac{q_j}{z-z_j}}^{-1},
\end{equation*}
with $\gamma$ as in \cref{prop:resequiv}. This might be of some interest for state-approximation methods like MRI and reduced basis.

\subsection{From state residual to output error estimation}
A naturally arising question is: does it make sense to drive a greedy algorithm for the \emph{output} transfer function $H$ with a residual estimator for the \emph{state} transfer function $G$? To answer this, we derive an $H$-oriented version of \cref{prop:resequiv}.

\begin{proposition}\label{prop:resequivH}
Let $\widetilde{H}$ be an interpolatory barycentric function as in \cref{eq:baryint}. The relative approximation error satisfies
\begin{equation}\label{eq:resequivH}
\varepsilon(\widetilde{H},z)=\Delta(z)\abs{\sum_{j=1}^S\frac{q_j}{z-z_j}}^{-1}.
\end{equation}
Above, the factor $\Delta$ is
\begin{equation}\label{eq:delta}
\Delta(z):=\frac{\dnorm{C(zE-A)^{-1}\widetilde{B}}}{\norm{C(zE-A)^{-1}B}+\delta}\leq \gamma\frac{\dnorm{(zE-A)^{-1}}_\star}{\norm{C(zE-A)^{-1}B}+\delta},
\end{equation}
with $\widetilde{B}\in\xC^{n\times m}$ and $\gamma>0$ depending only on $q_1,\ldots,q_S$, on $z_1,\ldots,z_S$, and on the system matrices. Also, $\norm{\cdot}_\star$ denotes the spectral norm of a matrix, i.e., its largest singular value.

Additionally, assume that \emph{all} eigenvalues of the pencil $(A,E)$ are poles of $H$, and that their multiplicities as eigenvalues equal their orders as poles, i.e., that $B$ and $C$ are in general position with respect to the generalized eigenspaces of $(A,E)$. Then, $\Delta$ is uniformly bounded from above by a global ($z$-independent) constant $\bar{\Delta}$, depending only on the system matrices and on $\gamma$: $\Delta(z)\leq\bar{\Delta}$ for all $z$.
\end{proposition}
\begin{proof}
First, we define $\widetilde{G}$ as in \cref{eq:statesurr}, by reusing the barycentric coefficients of $\widetilde{H}$. Since $H=CG$ and $\widetilde{H}=C\widetilde{G}$, we have
\begin{align*}
\varepsilon(\widetilde{H},z)=&\frac{\norm{C\left(\widetilde{G}(z)-G(z)\right)}}{\norm{H(z)}+\delta}\\
=&\frac{\norm{C(zE-A)^{-1}\left((zE-A)\widetilde{G}(z)-B\right)}}{\norm{H(z)}+\delta}\\
=&\frac{\norm{C(zE-A)^{-1}r(\widetilde{G},z)}}{\norm{H(z)}+\delta},
\end{align*}
with the residual $r$ defined as in the proof of \cref{prop:resequiv}. The first claim \cref{eq:resequivH} follows by applying \cref{eq:resvect}. Specifically, expression \cref{eq:delta} can be derived by setting
\begin{equation*}
\widetilde{B}=E\sum_{j=1}^Sq_jG(z_j)\quad\text{and}\quad \gamma=\norm{C}\dnorm{\widetilde{B}}.
\end{equation*}

Now it remains to bound $\Delta$ from above. For simplicity, we avoid a rigorous derivation, opting instead to only sketch the remainder of the proof.

We rely on the upper bound in the right-hand-side of \cref{eq:delta}. First, we assume that $z$ is far enough away from the poles of $H$ (i.e., the spectrum of $(A,E)$). Then, (i) the numerator is bounded from above and (ii) the denominator is bounded from below by $\delta$.

Otherwise, assume that $z$ is close to a pole $\lambda$ of $H$, whose order is $\nu$. Then, both numerator and denominator are large, since both have poles at $\lambda$. Specifically, by our assumption, the pole $\lambda$ has order $\nu$ for both numerator and denominator of $\Delta$. As such, both numerator and denominator have a leading term of order $\dabs{z-\lambda}^{-\nu}$:
\begin{equation*}
\Delta(z)\leq \gamma\frac{\dnorm{(zE-A)^{-1}}_\star}{\norm{C(zE-A)^{-1}B}+\delta}\approx \gamma\frac{\alpha/\dabs{z-\lambda}^\nu}{\beta/\dabs{z-\lambda}^\nu+\delta}\leq \gamma\frac{\alpha}{\beta}
\end{equation*}
for some $\alpha,\beta>0$, which are related to the norms of the residues of $\Phi$ (cf.~\cref{eq:ltit}) and of $H$ at $\lambda$. As such, $\Delta$ is uniformly bounded even on neighborhoods of the poles of $H$.
\end{proof}

According to \cref{prop:resequivH}, the \emph{error} for the \emph{output} transfer function behaves, with respect to $z$, almost like the \emph{residual} for the \emph{state} transfer function. The only additional $z$-dependent factor is $\Delta$, which plays the role of ``condition number'' in moving from residual to error. Unfortunately, such term depends on the spectral properties of the system matrices, and cannot be evaluated non-intrusively.

Also, $\Delta$ depends on $z$. As such, in general, the point with the largest error, namely, $z^\star_{S+1}:=\argmax_z\varepsilon(\widetilde{H},z)$, will be different from the point with the largest residual, namely, $z_{S+1}$ in \cref{eq:nextptQ}. However, \cref{prop:resequivH} also states that $\Delta$ is uniformly bounded. As such, we can hope that it does not vary too wildly. Intuitively, this ``hope'' is justified by the ``balanced'' structure of $\Delta$, whose numerator and denominator both contain the resolvent $(zE-A)^{-1}$, as shown in \cref{eq:delta}.

Specifically, if $\Delta$ is well-behaved (namely, if it only varies slightly over the considered frequency range), then we can expect the greedily built sets $\{z^\star_{j+1}\}_{j=1}^S$ and $\{z_{j+1}\}_{j=1}^S$ to be similar. Otherwise stated, we would like the (asymptotic) \emph{distribution} of greedy sample points to be approximately the same with both criteria. We rely on this requirement (stated only informally here) to justify the use of our criterion \cref{eq:nextptQ} to drive the greedy sampling even when the target is minimizing the output approximation error $\varepsilon$.

\subsection{Practical considerations for error maximization}
Locating the next frequency to be sampled (denoted by $z_{S+1}$ in \cref{algo}) is extremely efficient when using \cref{eq:nextptQ}. Indeed, one can easily apply the ``standard'' greedy technique based on a set of test frequencies, already described in \cref{sec:ideal}. To this aim, one chooses a fine set of $S'$ test frequencies \emph{a priori}. Then, at each greedy iteration, one evaluates the reciprocal of the indicator $\dabs{Q}=\dabs{\sum_{j=1}^Sq_j/(\cdot-z_j)}$ at all test frequencies and finds the smallest entry of the corresponding size-$S'$ vector. This is extremely efficient, since evaluating $\dabs{Q}$ only involves scalar operations.

Alternative approaches are also possible, based on an explicit \emph{continuous} minimization of $\dabs{Q}$ on the whole frequency range, as opposed to just on $S'$ discrete test frequencies. For this, it is crucial to observe that the barycentric denominator $Q$ is the ratio of two polynomials, one in Lagrangian form and one in nodal form:
\begin{equation*}
Q(z)=\frac{\sum_{j=1}^Sq_j\prod_{\substack{l=1\\l\neq j}}^S(z-z_l)}{\prod_{j=1}^S(z-z_j)}.
\end{equation*}
Accordingly, $\dabs{Q}$ attains its (global) minima at the roots of the numerator above, as long as they are distinct from the support points. Such roots $\widetilde{\lambda}$ can be efficiently computed by solving a generalized eigenvalue problem of size $S+1$:
\begin{equation}\label{eq:eigenarrow}
\begin{bmatrix}
0 & q_1 & q_2 & \cdots & q_S\\
1& z_1 & & & \\
1 & & z_2 & & \\
\vdots & & & \ddots &\\
1 & & & & z_S
\end{bmatrix}v=\widetilde{\lambda}\begin{bmatrix}
0 & \phantom{q_1} & \phantom{q_2} & \phantom{\cdots} & \phantom{q_S}\\
& 1 & & & \\
& & 1 & & \\
& & & \ddots & \\
& & & & 1
\end{bmatrix}v.
\end{equation}
See \cite{bary} for more details on this.

\section{Greedy termination criteria}\label{sec:term}
In choosing $\dabs{Q}^{-1}$ as indicator, we are neglecting the additional scaling factors $\gamma$ and $\Delta$ in \cref{eq:resequiv} and \cref{eq:resequivH}, respectively. The reason for doing this is that such constants are unavailable within the non-intrusive paradigm, as they depend on the high-fidelity system matrices. Consequently, so far, we have the ability to perform adaptive sampling, but we are unable to decide whether a given surrogate $\widetilde{H}$ is accurate enough, since we cannot estimate the magnitude of the error. As such, we do not yet have a termination criterion for \cref{algo}. In the next sections, we discuss some options for doing this.

\subsection{Maximum number or density of samples}
The most straightforward solution to the above-mentioned problem is simply not having an error-based stopping criterion at all. In this case, one has to stop the greedy iterations only based on the number and location of the sampled frequencies. For instance, one could simply set a maximum number of samples $S_{\textup{max}}$.

An alternative slightly more refined solution relies on setting a bound for the \emph{density} of samples: the algorithm ends when too many samples are taken too close to each other.

\subsection{Look-ahead error estimation}
As already mentioned in \cref{sec:ideal}, the ideal termination criterion \cref{eq:termination} is unfeasible, since it requires evaluating the high-fidelity response $H$ at all frequencies. In the same section, we have also described how the \emph{continuous} maximization can be replaced by a \emph{discrete} one, where the maximum is computed over a finite (but large) set of test frequencies.

As an even cheaper alternative, we can replace the continuous maximization in \cref{eq:termination} with a single point evaluation of the relative error:
\begin{equation}\label{eq:termsingle}
\varepsilon_1:=\varepsilon(\widetilde{H},z')<\textup{tol}.
\end{equation}
Above, $z'$ is some unexplored frequency, which should be carefully chosen so as to provide an insight on the largest approximation error: ideally, $\varepsilon_1\approx\max_z\varepsilon(\widetilde{H},z)$. To this aim, we propose to select $z'$ as the location that maximizes the error indicator $\dabs{Q}^{-1}$, i.e., $z':=z_{S+1}$. By \cref{prop:resequiv,prop:resequivH}, this value of $z'$ is located where the residual is largest, and where the error is heuristically expected to be largest. Note that, since $Q$ changes with $\widetilde{H}$, such test point $z'$ is different at each greedy iteration.

A second purpose is also served by choosing $z'=z_{S+1}$: evaluating the relative error at that point requires the sample $H(z_{S+1})$. This is exactly the sample that is computed anyway, whenever the termination condition is \emph{not} satisfied! This means that, by applying the 1-point criterion \cref{eq:termsingle}, only one expensive sample is wasted, namely, the final one used for testing. In effect, this means that the total ``error estimation'' cost in the greedy algorithm is just 1 high-fidelity sample, independently of the final number $S$ of samples.

\subsection{Look-ahead error estimation with memory}
The main drawback of the previous idea is its lack of robustness. Indeed, we are approximating a max-norm by a single sample, albeit carefully (but heuristically) chosen. Consequently, by using \cref{eq:termsingle}, one incurs the risk of underestimating the maximum approximation error. A simple way of increasing the robustness of this approach is introducing a ``memory'' term: we terminate the greedy sampling loop only if the termination criterion \cref{eq:termsingle} is satisfied by $N_{\textup{memory}}$ greedy iterations in a row, with $N_{\textup{memory}}\geq 2$ fixed in advance. We summarize this idea in \cref{algomemory}, which generalizes \cref{algo}.

\begin{algorithm}[tbh]
\caption{Surrogate modeling with greedy sampling and memory}\label{algomemory}
\begin{algorithmic}[1]
\State $\textbf{choose}\text{ the first sampled frequency }z_1\text{ and }\textbf{compute }H(z_1)$
\State $\textbf{choose}\text{ the memory depth }N_{\textup{memory}}\geq 1\text{ and }\textbf{set }n_{\textup{memory}}=0$\Comment{new!}
\For{$S=1,2,\ldots$}
\State $\textbf{build}\text{ the surrogate }\widetilde{H}\text{ from the available }S\text{ samples}$
\If{$\widetilde{H}\text{ is accurate enough}$}
	\State $\textbf{increase }n_{\textup{memory}}\text{ by }1$\Comment{new!}
	\If{$n_{\textup{memory}}\geq N_{\textup{memory}}$}\Comment{new!}
		\State $\textbf{return }\widetilde{H}$
	\EndIf\Comment{new!}
\Else\Comment{new!}
    \State$\textbf{reset }n_{\textup{memory}}=0$\Comment{new!}
\EndIf
\State $\textbf{find}\text{ the next point }z_{S+1}\text{ using a greedy criterion}$
\State $\textbf{compute }H(z_{S+1})\text{ and }\textbf{add }z_{S+1}\text{ to the sampled frequencies}$
\EndFor
\end{algorithmic}
\end{algorithm}

As we will show numerically, this strategy, despite its simplicity, can be surprisingly effective. Essentially, this approach relies on the idea that, although the indicator $\dabs{Q}^{-1}$ might occasionally poorly identify the location of the largest error, it does not do it persistently: one or two ``bad'' estimations of the error maximum are usually followed by a ``good'' one.

\subsection{Batch look-ahead error estimation}
Another simple way of increasing the robustness of the look-ahead approach is adding more than one sample point to estimate the max-norm of the error:
\begin{equation}\label{eq:termbatch}
\varepsilon_N:=\max_{i=1,\ldots,N}\varepsilon(\widetilde{H},z'_i)<\textup{tol}.
\end{equation}
Again, the test frequencies $z'_1,\ldots,z'_N$ should be carefully placed so as to well identify the largest approximation error. To this aim, one can choose local maxima of the error indicator $\dabs{Q}^{-1}$. As already mentioned, these maxima can be found by solving the low-dimensional root-finding problem \cref{eq:eigenarrow}. Note that, once again, such test points change from one greedy iteration to the next.

At each greedy iteration, this ``batch'' approach for error estimation requires $N$ extra expensive computations of the high-fidelity transfer function $H$. However, we observe that:
\begin{itemize}
\item thanks to \cref{prop:resequivH}, there are reasons to believe that a small or modest value of $N$ should be enough for a robust error estimation, since $\dabs{Q}^{-1}$ should be fairly reliable at estimating the error;
\item the extra expensive computations are relatively few ($\sim SN$), and they affect only the ``offline'' training phase of the surrogate.
\end{itemize}

Moreover, we note that, if the termination criterion \cref{eq:termbatch} is not satisfied (i.e., if the surrogate is still too inaccurate), \cref{algo} adds only a \emph{single} new sample point at $z_{S+1}$. However, a modification of \cref{algo} is possible, where all $N$ test points $z'_1,\ldots,z'_N$ (among which $z_{S+1}$ can also be found) are added to the training set. This comes at no additional computational cost, since their corresponding transfer function values are already available from the testing step. This can potentially remove the extra factor ``$N$'' in the cost of the offline training. For simplicity, we ignore this option here.

\subsection{Randomized error estimation}
The previous approach can be made even more robust, by choosing the test points $z'_1,\ldots,z'_N$ in \cref{eq:termbatch} independently of the surrogate. Specifically, one can simply select the test points randomly, according to some (quasi-)random distribution. This has the advantage of being a (probabilistically) rigorous estimator, with theoretical bounds \emph{in probability} \cite{random}. However, for a reliable error estimation, the number $N$ of test samples will generally be much larger than the one from the $Q$-based ``batch'' method from the previous section: we need enough samples to approximate the max-norm well. This can significantly increase the computational cost of the training procedure. To mitigate this effect, one should fix the $N$ test points once and for all, at the very beginning of the greedy sampling. In this way, the number of expensive test samples is a fixed overhead, which does not scale with $S$.

On top of the additional offline cost, the resulting approach has another drawback: the number $N$ must be chosen \emph{a priori}. At a fundamental level, this goes against the ``adaptivity'' brought forward by the present paper: how can $N$ be chosen well, without (intrusive) knowledge of the underlying system?

\section{Numerical examples}\label{sec:numerical}
In this section, we showcase our approach by applying it to three different MIMO examples from the SLICOT library\footnote{Niconet e.V, \url{http://slicot.org/20-site/126-benchmark-examples-for-model-reduction}.}. Our code is open-source, publicly available at \url{https://github.com/pradovera/greedy-loewner}.

In all our experiments, MIMO LF is used to build the surrogate transfer function $\widetilde{H}$ from the given data $\{(z_j,H(z_j))\}_{j=1}^S$. Specifically, we interpolate full $p\times m$ blocks of $H$ as in \cref{eq:mimolf}, instead of applying the more ``classical'' MIMO LF based on tangential interpolation \cite{LF2}.

In all cases, we use our proposed greedy method from \cref{sec:greedy}: at each greedy iteration, to find $z_{S+1}$, the indicator $\dabs{Q}^{-1}$ is maximized over a fine grid of $10^4$ geometrically spaced test frequencies. We always set $\textup{tol}=10^{-3}$ and $\delta=10^{-8}$.

\subsection{Example \#1: \texttt{MNA\_4}}
We consider the impedance formulation of a 4-port electrical circuit, obtained by modified nodal analysis. This results in a descriptor system whose sizes are $n=980$ and $p=m=4$. We look at a range of imaginary frequencies $z\in[3\cdot 10^4, 3\cdot 10^9]\imath$, with $\imath$ the imaginary unit. The system response is shown in \cref{fig:mna} (top), and, to approximate it, we apply our proposed greedy method. As termination condition, we test two options: look-ahead error estimation and randomized error estimation. For the latter, we take $N=100$ test frequencies $z'_1,\ldots,z'_N$, drawn from a log-uniform distribution on the frequency range. As already mentioned, the random estimator comes with one main computational drawback: the additional overhead of computing the $N$ test samples.

\begin{figure}[tb]
	\centering
	\includegraphics{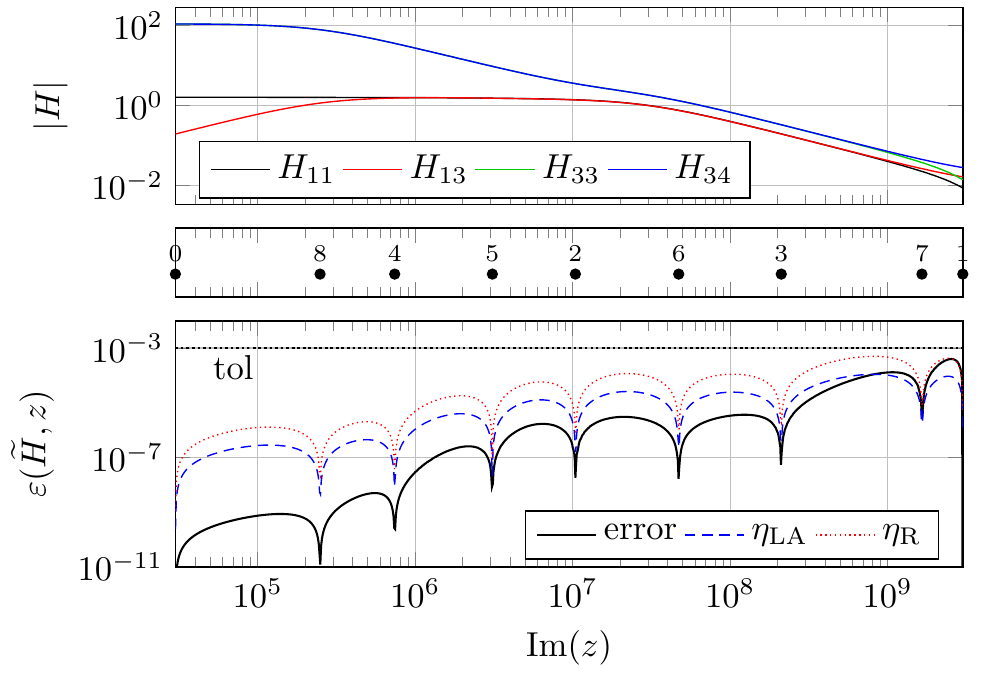}
    \caption{Results for the \texttt{MNA\_4} example. Top plot: some entries of the exact transfer function $H$. Middle plot: adaptively sampled frequencies with their index number (in order of addition). Bottom plot: relative approximation error and error estimators.}
    \label{fig:mna}
\end{figure}

Both approaches converge in 9 greedy iterations. The sampled frequencies are chosen according to the same indicator, so that the surrogates obtained with the two methods are actually identical. Only the error estimates used for the termination condition are different. We do not display the surrogate transfer functions $\widetilde{H}$ in \cref{fig:mna} (top), since they are indistinguishable from the exact $H$. We show the sampled frequencies in \cref{fig:mna} (middle). In \cref{fig:mna} (bottom), we also show the approximation error, which is uniformly below the tolerance.

Moreover, we test how accurate the two methods are at estimating the error. To this aim, we consider the following quantities:
\begin{equation*}
\eta_{\textup{LA}}(z)=\varepsilon_1\abs{\frac{Q(z_{S+1})}{Q(z)}}
\end{equation*}
for the look-ahead method, cf.~\cref{eq:termsingle}, and
\begin{equation}\label{eq:etaR}
\eta_{\textup{R}}(z)=\varepsilon_N\abs{\frac{Q(z'_{j^\star})}{Q(z)}},\quad\text{where }z'_{j^\star}\text{ is the argmax in \cref{eq:termbatch}},
\end{equation}
for the randomized method. We can see that both of them are simply rescaled versions of the indicator $\dabs{Q}^{-1}$. The scaling factors are chosen so that they interpolate the error estimate that each method employs: $\eta_{\textup{LA}}(z_{S+1})=\varepsilon_1$ for the look-ahead method, and $\eta_{\textup{R}}(z'_{j^\star})=\varepsilon_N$ for the randomized method. Effectively, $\eta_{\textup{LA}}$ and $\eta_{\textup{R}}$ are the error \emph{estimators} that the two methods use. (We call them so because, as opposed to \emph{indicators}, they also include information on the error magnitude.)

We compare the estimators with the exact error in \cref{fig:mna} (bottom). We can observe that both approximate the error quite nicely. Specifically, we know from \cref{prop:resequivH} that any discrepancy between error and estimators is due to the factor $\Delta$. We note that $\eta_{\textup{LA}}$ underestimates the error for large frequencies ($z>10^9$). In comparison, $\eta_{\textup{R}}$ is more robust, and never underestimates the error.

As a final remark, we note that the ``look-ahead'' approach was able to properly identify the ``simplicity'' of the system: less than $10$ high-fidelity samples suffice in identifying $H$ to the desired accuracy. The same can be said about the ``randomized'' approach, albeit to a lesser degree, due to the extra overhead related to the $N$ testing samples. In comparison, had a non-adaptive approach like VF or AAA been chosen, one would have probably elected to take (much) more than 10 samples of $H$, incurring a higher training cost.

\subsection{Example \#2: \texttt{tline}}
\begin{figure}[p]
	\centering
	\includegraphics{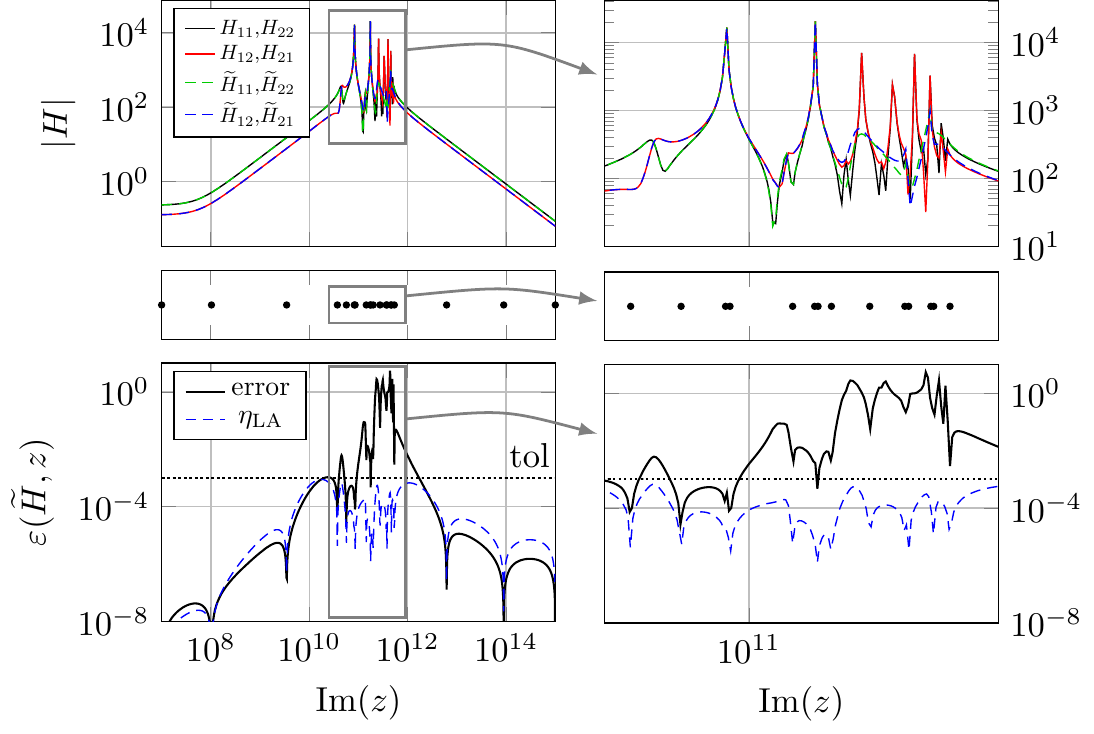}
    \caption{Results for the \texttt{tline} example. Top row: exact transfer function $H$. Middle row: adaptively sampled frequencies. Bottom row: relative approximation error and error estimator. The same results are shown in the two columns of plots: the right figures are zooms on the ``most interesting'' frequencies.}
    \label{fig:tline}
\end{figure}

\begin{figure}[p]
	\centering
	\includegraphics{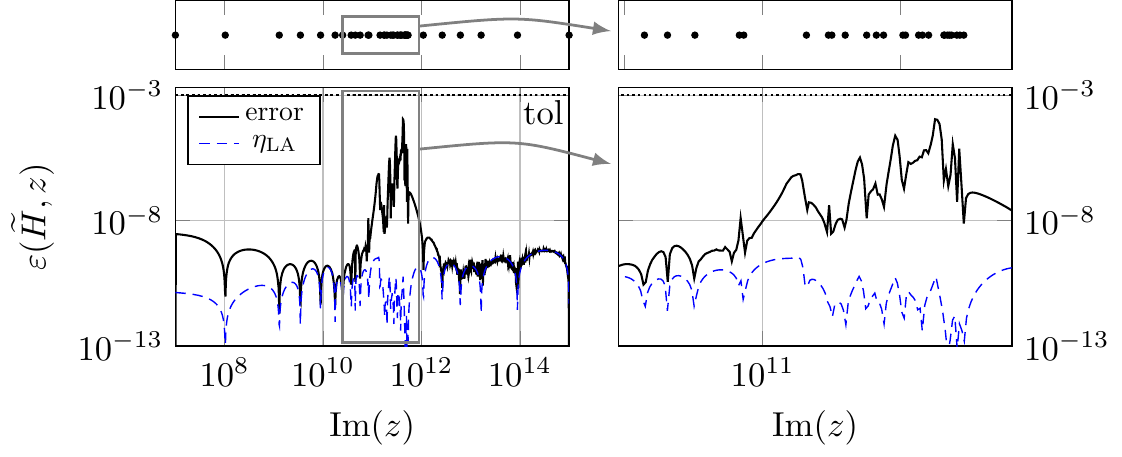}
    \caption{Results for the \texttt{tline} example with memory. Top row: adaptively sampled frequencies. Bottom row: relative approximation error and error estimator. The same results are shown in the two columns of plots: the right figures are zooms on the ``most interesting'' frequencies.}
    \label{fig:tline_x}
\end{figure}

\begin{figure}[p]
	\centering
	\includegraphics{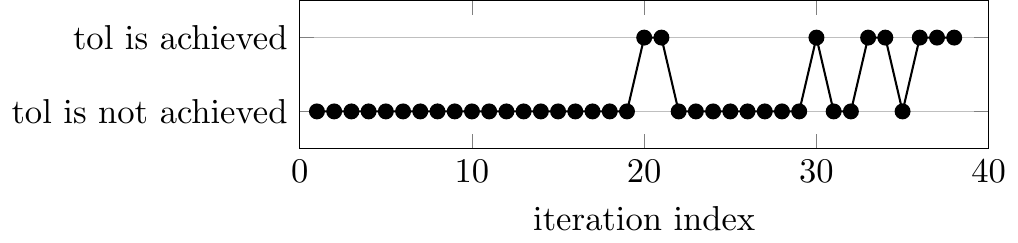}
    \caption{Evolution of the greedy termination flag (based on the 1-point look-ahead estimator) for the \texttt{tline} example with depth-3 memory. Note: the desired $\textup{tol}$ is achieved also at all iterations past the $38$\textsuperscript{th} one.}
    \label{fig:tline_f}
\end{figure}

We consider the impedance formulation of a 2-port electrical circuit modeling a transmission line. The resulting descriptor system has sizes $n=256$ and $p=m=2$. We look at a range of imaginary frequencies $z\in[10^7, 10^{15}]\imath$. The system response is shown in \cref{fig:tline} (top). We can observe strong resonating dynamics in the frequency band $[10^{10},10^{12}]\imath$. To approximate $H$, we apply our greedy method, using the simple look-ahead criterion as termination condition.

The greedy algorithm converges in 20 iterations. We display the corresponding surrogate $\widetilde{H}$ in \cref{fig:tline} (top) and the adaptively chosen sample points in \cref{fig:tline} (middle). We can see that the approximation quality is rather poor for some frequencies: $\widetilde{H}$ fails to identify and approximate several poles of $H$. We can understand the reason for this in \cref{fig:tline} (bottom), where we display both the exact error and its estimate $\eta_{\textup{LA}}$: the error is grossly underestimated over a substantial portion of the frequency range. From a theoretical viewpoint, this appears to be due to the high concentration of resonating frequencies around $z=3\cdot 10^{11}\imath$, which makes the factor $\Delta$ in \cref{eq:resequivH} vary across several orders of magnitude. In our specific experiment, the error estimator fails to identify the magnitude of such factor. For this reason, the maximum of the error estimator is not close to the maximum of the actual error, resulting in an early termination of the greedy sampling.

In \cref{sec:term}, we presented several ways of counteracting this issue. Here, we choose the simplest one: introducing a memory term. Specifically, we set $N_{\textup{memory}}=3$ in \cref{algomemory}. With this modification, the algorithm stops after $38$ iterations instead of just $20$.

The corresponding results are in \cref{fig:tline_x}. In \cref{fig:tline_x} (bottom) we see that, again, the error is underestimated over much of the frequency range. However, the addition of the memory term artificially increased the number of greedy iterations. Thanks to this effect, the approximation error is now uniformly below the tolerance.

Also, in \cref{fig:tline_f} we display the ``truth value'' of the termination criterion (i.e., whether \cref{eq:termsingle} is true or false) as the greedy iterations proceed. We can see that, despite a few short ``false-negative'' streaks, the estimator seems to be able to recognize and correct its own mistakes.

In this benchmark, we note that adaptivity is very effective at identifying the sub-range of frequencies where the response varies the most. If we had employed uniform or log-uniform sampling, which are the golden standards for VF and AAA, we estimate that at least $500$ samples would have been necessary to attain the prescribed accuracy.

\subsection{Example \#3: \texttt{iss}}
Lastly, we test our algorithm on a discrete frequency-domain structural response model of a module of the international space station. The resulting system has sizes $n=270$ and $p=m=3$. We look at a range of imaginary frequencies $z\in[0.1, 50]\imath$. The system response is shown in \cref{fig:iss} (top). We can observe that this is essentially a more complicated version of the \texttt{tline} example: the resonating frequencies are more, and the resonating dynamics have less uniform strengths.

We apply our method with look-ahead error estimation, with ``memory depth'' $N_{\textup{memory}}=3$. The algorithm carries out $100$ greedy iterations before terminating. We show the corresponding results in \cref{fig:iss} (bottom). As in the previous case, the error is underestimated at many frequencies.

As a way to improve the quality of the estimator, we test the ``batch error estimation'' idea from \cref{sec:term}: at each greedy iteration, we compute the exact error at the locations $z'_1,\ldots,z'_N$ of the $N=5$ largest local maxima of the error indicator $\dabs{Q}^{-1}$. Then we take the maximum of such $5$ errors and we compare it with $\textup{tol}$ to determine whether to continue or not.

This criterion is more strict than the ``1-point'' one used before. As a result, $12$ extra greedy iterations are carried out. The corresponding error estimator $\eta_{\textup{BLA}}$, defined as in \cref{eq:etaR}, is shown in \cref{fig:issb}. We can observe that the error bound is much more reliable. The resulting approximation error is almost uniformly below the tolerance. Note, however, that, compared to the simple ``1-point'' method, the ``batch'' method had to compute approximately $(N-1)S$ extra high-fidelity samples for testing purposes, effectively increasing the offline training cost by a factor $N$.

\begin{figure}[tb]
	\centering
	\includegraphics{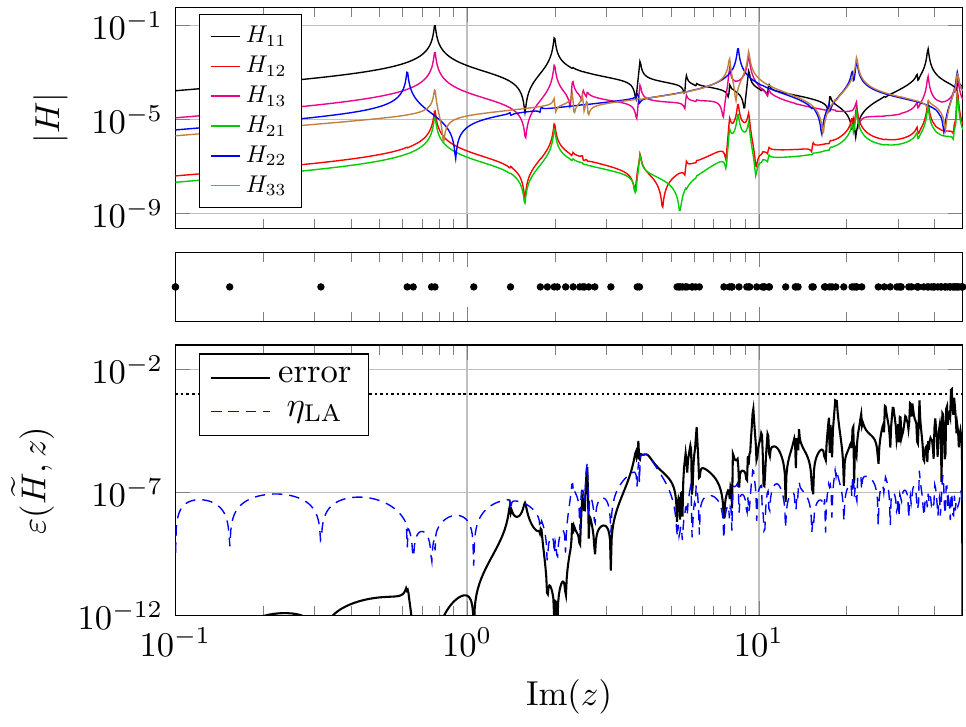}
    \caption{Results for the \texttt{iss} example. Top plot: some entries of the exact transfer function $H$. Middle plot: adaptively sampled frequencies. Bottom plot: relative approximation error and error estimator.}
    \label{fig:iss}
\end{figure}

\begin{figure}[tb]
	\centering
	\includegraphics{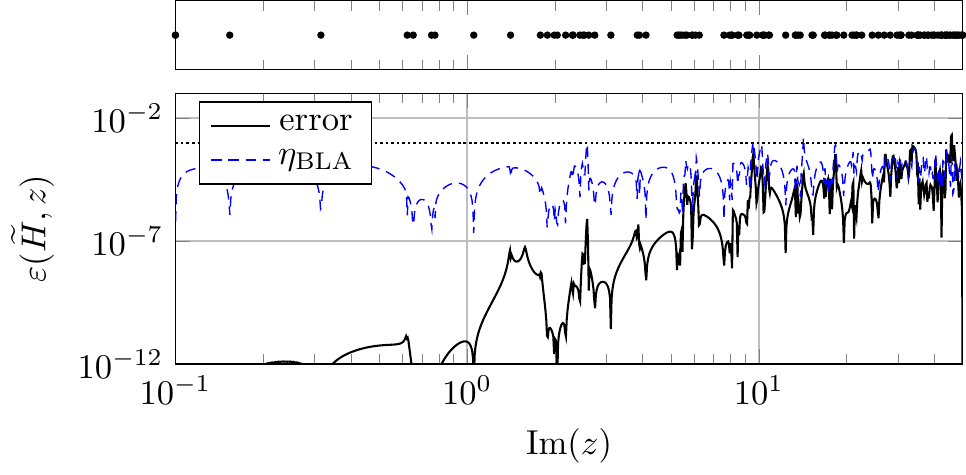}
    \caption{Results for the \texttt{iss} example with batch error estimation. Top plot: adaptively sampled frequencies. Bottom plot: relative approximation error and error estimator.}
    \label{fig:issb}
\end{figure}

As a third alternative, we also consider the ``randomized error estimation'' idea, using $N=100$ test frequencies, log-uniformly spaced over the frequency range. Similarly to the original 1-point indicator, the greedy algorithm terminates too early (this time after 87 iterations), with the error being underestimated by the estimator $\eta_{\textup{R}}$, cf.~\cref{eq:etaR}. A larger value of $N$ might improve the reliability of the randomized estimator, but, of course, this information would not have been available \emph{a priori}.

Finally, we compare the three error estimation strategies, by looking at the behavior of the estimators $\eta_{\textup{LA}}$, $\eta_{\textup{BLA}}$, and $\eta_{\textup{R}}$ as the greedy iterations proceed. The results are shown in \cref{fig:iss_c}. We can observe that the look-ahead estimators $\eta_{\textup{LA}}$ and $\eta_{\textup{BLA}}$ often coincide, since the single test frequency used for $\eta_{\textup{LA}}$ in \cref{eq:termsingle} belongs to the batch of test frequencies used for $\eta_{\textup{BLA}}$ in \cref{eq:termbatch}. Notably, both look-ahead estimators appear to oscillate as the iterations proceed. Ultimately, the oscillations are due to the fact that the indicators rely on evaluations of the error at test frequencies that vary at each iteration. The oscillations are particularly severe for the 1-point estimator, which also has several ``downward spikes'' below the tolerance level, symptoms of the underestimation of the error magnitude.

On the other hand, the randomized estimator $\eta_{\textup{R}}$ seems to be more stable and less oscillatory. This different behavior is justified by the fact that $\eta_{\textup{R}}$ relies on relative errors computed at test frequencies fixed in advance, which do not change throughout the greedy iterations. From this point of view, we can expect the randomized estimator to achieve a higher robustness, compared to the look-ahead ones. Still, this increased robustness of $\eta_{\textup{R}}$ does not necessarily imply reliability, as evidenced by the greedy algorithm terminating too early due to underestimated errors.

\begin{figure}[tb]
	\centering
	\includegraphics{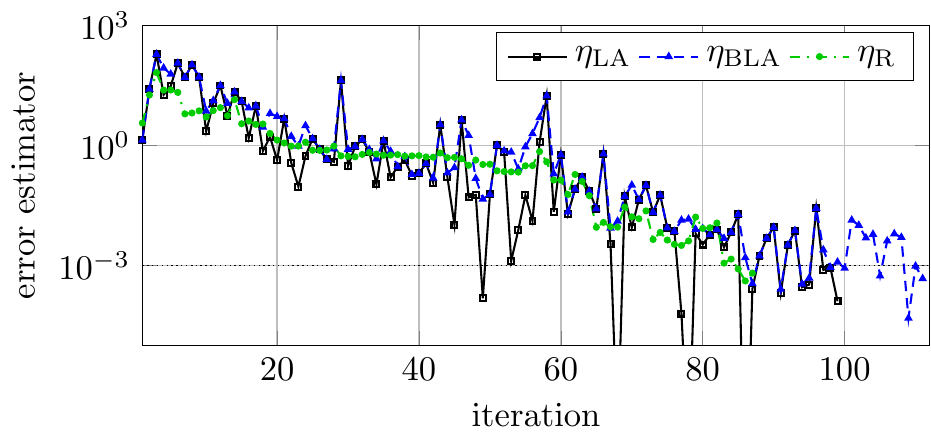}
    \caption{Evolution of the error estimator for the \texttt{iss} example.}
    \label{fig:iss_c}
\end{figure}

\section{Discussion and conclusions}\label{sec:conclusion}
In this work, we have proposed a non-intrusive error indicator for the rational approximation of frequency responses. The indicator is based on the denominator of the rational surrogate in barycentric form, and allows building rational approximations of transfer functions in an incremental greedy fashion, even in non-intrusive settings. The issue of choosing the greedy termination criterion was also discussed, and several approaches were presented for this purpose.

On one hand, our contribution can be compared to other ``adaptive'' approaches in the standard sense, like AAA. With respect to these methods, ours has the great advantage of not requiring an initial fine sampling. This leads to a lower cost in the training phase, making our method superior especially when evaluating the frequency response is expensive. On the other hand, our error indicator can be supported by a partial theory. This is an improvement over comparable state-of-the-art heuristic error indicators, like those proposed in \cite{gLF1,gLF2,gLF3,gLF4}.

Note that we do not claim our adaptive method to build a rational surrogate having \emph{minimal} order, e.g., in the sense of \cite{opt1}, nor \emph{optimality} properties, e.g., in the sense of \cite{opt2}. In particular, the final number of sample points may be larger than strictly necessary to achieve the target tolerance. This is due to two effects: (i) our error indicator is not an estimator, so that each new sample point may be placed sub-optimally, compared to the actual maximum of the error, which is out of reach in a non-intrusive setting (and also in most intrusive ones); (ii) greedy approaches, even when using intrusive error \emph{estimators}, are only guaranteed to yield a \emph{sub-}optimal approximation \cite{greedy}. In our view, thanks to its generality, our approach has significance despite its sub-optimality, since optimality is often incompatible with useful algorithmic features like non-intrusiveness and sampling efficiency.

Concerning the value of $\delta$ in the definition of the adjusted relative error $\varepsilon$, we mention that our discussion generalizes to $\delta=0$, which corresponds to the usual relative approximation error. However, in this case, $\Delta$ may be unbounded near the zeros of $H$. Moreover, $\delta\to\infty$ is also allowed, which, by scaling $\textup{tol}\mapsto\delta\cdot\textup{tol}$, corresponds to a greedy algorithm based on the \emph{absolute} approximation error as opposed to the relative one. In this case, $\Delta$ may be unbounded near the poles of $H$.

In the scope of improving the stability and reliability of our error estimator, future research directions include: (i) a more careful study and/or estimation of the $z$-dependent $\Delta$ factor in the error bound, (ii) the development of more reliable termination criteria, possibly improving the ``batch'' approach, and (iii) a comparison of our proposed method with the above-cited state-of-the-art alternatives, in terms of performance, stability, and effectiveness.

Finally, we mention the possibility of applying the proposed greedy method within the context of \emph{parametric} problems, for instance by following the non-intrusive parametric MOR framework introduced in \cite{pMRI}.

\backmatter

\bmhead{Acknowledgments} The author is grateful to Fabio Nobile for valuable discussions.

\bibliography{bibliography}

\end{document}